\documentclass{amsart}
\usepackage{amsfonts, amsmath, amssymb,centernot,tikz-cd}
\usepackage[margin=1.5cm]{geometry}
\newcommand{\Sym}{\mathop{\mathrm{Sym}}}
\newcommand{\Aut}{\mathop{\mathrm{Aut}}}
\newcommand{\Out}{\mathop{\mathrm{Out}}}
\newtheorem{prop}{Proposition}[section]

\newtheorem{thrm}[prop]{Theorem}
\newtheorem{conj}[prop]{Conjecture}

\newtheorem{example}[prop]{Example}

\newtheorem{lemma}[prop]{Lemma}

\def\cent#1#2{{\bf C}_{{#1}}{{(#2)}}}
\def\norm#1#2{{\bf N}_{{#1}}{{(#2)}}}

\begin{document}

\title[Two point stabilizers]{A generalization of Sims conjecture for finite primitive groups and two point stabilizers in primitive groups}  
\author{Pablo Spiga}
\address{Dipartimento di Matematica e Applicazioni, University of Milano-Bicocca, Via Cozzi 55, 20125 Milano, Italy} 
\email{pablo.spiga@unimib.it}


\begin{abstract}
In this paper we propose a refinement of Sims conjecture concerning the cardinality of the point stabilizers in finite primitive groups and we make some progress towards this refinement.

In this process, when dealing with primitive groups of diagonal type, we construct a finite primitive group $G$ on $\Omega$ and two distinct points $\alpha,\beta\in \Omega$ with $G_{\alpha\beta}\unlhd G_\alpha$ and  $G_{\alpha\beta}\ne 1$, where $G_{\alpha}$ is the stabilizer of $\alpha$ in $G$ and $G_{\alpha\beta}$ is the stabilizer of $\alpha$ and $\beta$ in $G$. In particular, this example gives an answer to a  question raised independently by Peter Cameron in~\cite{Cameron} and by Alexander Fomin in the Kourovka Notebook~\cite[Question~9.69]{KN}.
\end{abstract}

\maketitle

\section{Introduction}
Let $G$ be a finite primitive group acting on a set $\Omega$ and let $\alpha\in\Omega$. The \textbf{\em subdegrees} of $G$ are the lengths of the orbits of the \textbf{\em point stabilizer} $G_\alpha:=\{g\in G\mid \alpha^g=\alpha\}$ on $\Omega$. 

 Given a subdegree $d$ of $G$, there is no  bound on the degree $|\Omega|$ of $G$, as a function of $d$ {\em only}. For example, for any prime $p$, the dihedral group of order $2p$ has a faithful primitive permutation representation of degree  $p$ with subdegree $d=2$. Despite this, if a finite primitive group $G$ has  a small subdegree, then the structure of $G_\alpha$ is rather restricted, see for instance~\cite{Quirin,Sims,Wong} and the much more recent results in~\cite{FGLPRV,LiLuMa}. Following the investigations on the cases $d=3$ and $d=4$, Charles Sims~\cite{Sims} was lead to conjecture the following. 

\begin{thrm}[Sims conjecture]\label{thrm:sims}
There is a function $f:\mathbb{N}\to \mathbb{N}$ such that, if $G$ is a finite primitive group with a suborbit of length $d>1$, then the stabilizers have order at most $f(d)$.
\end{thrm}
This theorem was proved by Cameron, Praeger, Saxl and Seitz~\cite{CPSS}  using the O'Nan-Scott Theorem and the  recently (at the time) announced classification of finite simple groups. The CFSGs spurred a new vitality in the old subject of finite permutation groups and this result was one of the first major applications of the CFSGs, see also~\cite{Peter}.
  
In this paper we propose a strengthening of Sims conjecture that in our opinion  captures the structure of point stabilizers of primitive groups to a finer degree. As well as Sims conjecture, our conjecture can be phrased in purely group-theoretic terminology, but it is better understood borrowing some terminology from graph theory.

Let $G$ be a finite primitive group acting on a set $\Omega$ and let $\alpha$ and $\beta$ be two elements of $\Omega$. The \textbf{\em orbital graph} $\Gamma$ determined by the ordered pair $(\alpha,\beta)$, is the directed graph with vertex set $\Omega$ and with arc set 
$(\alpha,\beta)^G:=\{(\alpha^g,\beta^g)\mid g\in G\}$. 
Clearly, $G$ is  a group of automorphisms of $\Gamma$ acting primitively on its vertex set and 
$\Gamma$ is undirected if and only if the \textbf{\em orbital} $(\alpha,\beta)^G$ is \textbf{\em self-paired}, that is, $(\alpha,\beta)^G=(\beta,\alpha)^G$. We denote by 
$\Gamma^+(\gamma):=\{\delta\mid (\gamma,\delta)\in (\alpha,\beta)^G\}$ and by $\Gamma^-(\gamma):=\{\delta\in\Omega\mid (\delta,\gamma)\in (\alpha,\beta)^G\}$ the \textbf{\em out-neighborhood} and the \textbf{\em in-neighborhood}, respectively, of the vertex $\gamma$ of $\Gamma$. As $\Omega$ is finite, it follows that the \textbf{\em out-valency} and the \textbf{\em in-valency} of $\Gamma$ are equal, that is, $|\Gamma^+(\gamma)|=|\Gamma^-(\gamma)|$ for every $\gamma\in \Omega$. Moreover, if we denote by $d$ this valency, then $d$ equals the cardinality of the suborbit $\beta^{G_{\alpha}}=\Gamma^+(\alpha)$. In particular, $d$ is a subdegree of $G$.

Given two vertices $\alpha$ and $\beta$ as above, we write $G_{\alpha\beta}:=G_\alpha\cap G_\beta$. Moreover, we denote with $$G_{\alpha}^{+[1]}:=\bigcap_{\delta\in\Gamma^+(\alpha)}G_{\alpha\delta} \,\hbox{ and }\, G_\beta^{-[1]}:=\bigcap_{\delta\in\Gamma^-(\beta)}G_{\delta\beta},$$ the kernel of the action of $G_\alpha$ on $\Gamma^+(\alpha)$ and of $G_\beta$ on $\Gamma^{-}(\beta)$, respectively. Observe that the notation $G_\alpha^{+[1]}$ and $G_\beta^{-[1]}$ is slightly misleading because it does not show the dependency of this subgroup of $G$ from the graph $\Gamma$, however for not making the notation too cumbersome to use we prefer not to  attach the label ``$\Gamma$'' in the notation for $G_\alpha^{+[1]}$ and $G_\beta^{-[1]}$.

 Furthermore, we denote by $G_\alpha^{\Gamma^+(\alpha)}\cong G_\alpha/G_{\alpha}^{+[1]}$ and by $G_{\beta}^{\Gamma^-(\beta)}\cong G_\beta/G_{\beta}^{-[1]}$, the permutation group induced by $G_\beta$ on $\Gamma^+(\alpha)$ and on $\Gamma^-(\beta)$, respectively. The groups $G_\alpha^{\Gamma^+(\alpha)}$ and $G_{\beta}^{\Gamma^-(\beta)}$ are (not necessarily isomorphic) permutation groups of degree $d$ and they are sometimes refereed to as the \textbf{\em local groups}, see~\cite{PSV1,PSV2} where this terminology is particularly suited.

Using the notation that we have established above, Sims conjecture claims that, when $d>1$, the cardinality of $G_\alpha$ is bounded above by a function of the valency $d$ of the orbital graph $\Gamma$. In other words, the order of the vertex stabilizer $G_\alpha$ is bounded above by a function of the local group $G_{\alpha}^{\Gamma^+(\alpha)}$. (Actually, the order of the vertex stabilizer is bounded above simply by a function of the degree of the local group.) Broadly speaking, we wish to make a step further and we wonder whether, besides some families that can be explicitly classified, the order of the arc stabilizer $G_{\alpha\beta}$ is bounded above by a function of the order of the point stabilizer $G_{\alpha\beta}^{\Gamma^+(\alpha)}$ of the local group $G_{\alpha}^{\Gamma^+(\alpha)}$. More precisely, we propose the following conjecture:
\begin{conj}\label{q:1}
There exists a function $g:\mathbb{N}\to \mathbb{N}$ such that, if $G$ is a finite primitive group with a suborbit $\beta^{G_\alpha}$ of length $d>1$, then either
\begin{description}
\item[(i)] $G_{\alpha\beta}$ has order at most $g(|G_{\alpha\beta}:G_{\alpha}^{+[1]}|)$, or
\item[(ii)]$G$ is in a well-described and well-determined list of exceptions.
\end{description}
\end{conj}
It is quite unfortunate, that we cannot omit alternative~${\bf (ii)}$ in Conjecture~\ref{q:1}. Examples in this direction are intricate to construct and we give an example later in this paper, see Example~\ref{examplemrone}. A positive solution to Conjecture~\ref{q:1} can be seen as a strengthening of Sims conjecture, this is easy to see but we postpone the  proof to Section~\ref{sub:21}. However, a positive answer to Conjecture~\ref{q:1} is a  stronger result than Sims conjecture: we refer the reader to Example~\ref{example} to see this and   to capture the idea behind our question. 
In this paper, we give some evidence towards the veracity of Conjecture~\ref{q:1}.

A particularly interesting case for our conjecture is when $G_{\alpha\beta}\unlhd G_\alpha$. In this case, $G_\alpha^{+[1]}=G_{\alpha\beta}$, that is, the local group $G_\alpha^{\Gamma^+(\alpha)}$ is regular. In this case, it was asked by Peter Cameron~\cite{Cameron} whether $G_{\alpha\beta}=1$, that is, whether the whole vertex stabilizer $G_\alpha$ acts regularly on the out-neighborhood $\Gamma^+(\alpha)$. This question was also proposed by Fomin in the Kourovka Notebook~\cite[Question~9.69]{KN}. Remarkable evidence towards the veracity of the question of Cameron and Fomin is given by Konygin~\cite{AVK0,AVK,AVK1,AVK2}, however to the best of our knowledge this question is still open. To some extent, Conjecture~\ref{q:1} can be seen as a generalization of the question of Cameron and Fomin, allowing $G_{\alpha\beta}\centernot\unlhd G_\alpha$, but relaxing the conclusion when $G_{\alpha\beta}\unlhd G_\alpha$. By investigating Conjecture~\ref{q:1} for primitive groups of diagonal type, we construct a finite primitive group with $G_{\alpha\beta}\unlhd G_\alpha$ and $G_{\alpha\beta}\ne 1$, thus giving an answer in the negative to the question of Cameron and Fomin. We present this example in Section~\ref{sec:example}.

\section{Basic results}

\subsection{Relations between Sims conjecture and Conjecture~\ref{q:1}}\label{sub:21}
\begin{lemma}\label{lemma2}
Suppose that  Conjecture~$\ref{q:1}$ holds true. Then, any group satisfying part $\mathbf{(i)}$ of Conjecture~$\ref{q:1}$, satisfies Sims conjecture.
\end{lemma}
\begin{proof}Let $g$ be the function arising from a positive solution of Conjecture~\ref{q:1}.
We define a function $g':\mathbb{N}\to \mathbb{N}$ by setting $$g'(d):=\max\{g(|H_\delta|)\mid H\textrm{ transitive on }\{1,\ldots,d\}, \delta\in\{1,\ldots,d\}\}.$$

Now, let $G$ be a finite primitive group satisfying part {\bf (i)} of Conjecture~\ref{q:1}, let $\beta^{G_\alpha}$ be a suborbit of $G$ of cardinality $d>1$ and let $\Gamma$ be the corresponding orbital graph. Set $H:=G_\alpha^{+\Gamma(\alpha)}$ and observe that $H$ is a transitive permutation group on $\Gamma^+(\alpha)$ of degree $d$. The stabiliser of a point in $H$ is $G_{\alpha\beta}^{\Gamma^+(\alpha)}$. Since $G$ satisfies  Conjecture~\ref{q:1}~{\bf (i)}, we have 
$$|G_{\alpha\beta}|\le g(|G_{\alpha\beta}^{\Gamma^+(\alpha)}|)\le g'(d).$$ Now, $|G_\alpha|=|G_\alpha:G_{\alpha\beta}||G_{\alpha\beta}|=d|G_{\alpha\beta}|\le dg'(d)$. Therefore, Sims conjecture holds for $G$ by taking $f(d):=dg'(d)$.
\end{proof}

\begin{example}\label{example}{\rm 
It is interesting to consider finite primitive groups having a suborbit $\beta^{G_\alpha}$ of odd cardinality $d$ with $G_\alpha$ acting as a dihedral group on $\beta^{G_\alpha}$. In this case, $|G_{\alpha}:G_{\alpha\beta}|=d$ and  $|G_{\alpha\beta}:G_{\alpha}^{+[1]}|=2$. It follows from the main result in~\cite{Sami} (and also from the work of Verret on $p$-subregular actions~\cite{Verret,Verret1}), that $|G_{\alpha\beta}|$ divides $16$. In particular, $|G_{\alpha\beta}|\le 16$. Observe that this upper bound on $|G_{\alpha\beta}|$ does not depend on $d$.

In particular, in this example, Sims conjecture requires bounding $|G_\alpha|$ as a function of $d$. However, Conjecture~\ref{q:1} is more demanding: since $|G_{\alpha\beta}:G_{\alpha}^{+[1]}|=2$ is a constant, Conjecture~\ref{q:1} demands either regarding $G$ is an exception or bounding $|G_{\alpha\beta}|$ from above with an absolute constant. Luckily, in this case, $G$ is not an exception, because from~\cite{Verret,Verret1}, $|G_{\alpha\beta}|\le 16$.}
\end{example}

Lemma~\ref{lemma2} shows that, if Conjecture~\ref{q:1} holds true, then it might be possible proving Sims conjecture by checking (possibly with a direct case-by-case inspection) the groups falling in part {\bf (ii)}.

\subsection{The O'Nan-Scott theorem and our investigation}\label{sub:22}
The modern key for analyzing a finite primitive permutation group $G$ is to
study the \textit{socle} $N$ of $G$, that is, the subgroup generated
by the minimal normal subgroups of $G$. The socle of an arbitrary
finite group is isomorphic to the non-trivial direct product of simple
groups; moreover, for finite primitive groups these simple groups are
pairwise isomorphic. The O'Nan-Scott theorem describes in details the
embedding of $N$ in $G$ and collects some useful information about the
action of $N$. In~\cite[Theorem]{LPSLPS} five types of primitive groups
are defined (depending on the group- and action-structure of the
socle), namely HA (\emph{Affine}), AS (\emph{Almost Simple}), SD
(\emph{Simple Diagonal}), PA (\emph{Product Action}) and TW
(\emph{Twisted Wreath}), and it is shown that  every primitive group
belongs to exactly one of these types. We remark that in~\cite{C3}
this subdivision into types is refined, namely the PA type
in~\cite{LPSLPS} is partitioned in four parts, which are called HS (\emph{Holomorphic simple}), HC (\emph{Holomorphic compound}), CD
(\emph{Compound Diagonal}) and PA.  For what follows, we find it convenient to use this subdivision into eight types of the finite primitive primitive groups.

In this paper we investigate Conjecture~\ref{q:1} using the O'Nan-Scott theorem.

\section{Primitive groups of HA, TW, HS and HC type}
\begin{prop}\label{prop:1}
Let $G$ be a transitive group on $\Omega$ containing a normal regular subgroup, let $\Gamma$ be a connected digraph with vertex set $\Omega$ and  left invariant by the action of $G$. Then $G_\alpha^{+[1]}=1$, for every $\alpha\in \Omega$. In particular, Conjecture~$\ref{q:1}~{\bf (i)}$ holds true when $G$ has O'Nan-Scott type $\mathrm{HA}$, $\mathrm{TW}$, $\mathrm{HS}$ and $\mathrm{HC}$ and the function $g$ can be taken so that $g(n):=1$, for every $n\in \mathbb{N}$.
\end{prop}
\begin{proof}
Let $\alpha\in \Omega$. Let $N$ be a regular normal subgroup of $G$ and let $H:=G_\alpha$. Then, $G$ is the semidirect product of $N$ by $H$ (that is, $G=NG_\alpha$, $N\cap G_\alpha=1$ and $G=N\rtimes H$) and the action of $G$ on $\Omega$ is permutation equivalent to the ``affine'' action of $G$ on $N$, where $N$ acts on $N$ by right multiplication and where $H$ acts on $N$ by conjugation. In the rest of the proof, we use this identification. In particular, under this equivalence, $\alpha\in\Omega$ corresponds to $1\in N$. For each $\beta\in \Gamma^+(\alpha)$, we let $n_\beta$ be the element of $N$ corresponding to $\beta$, that is, $\alpha^{n_\beta}=\beta$.

Then $G_{\alpha\beta}=\cent H {n_\beta}$ and 
$$G_{\alpha}^{+[1]}=\bigcap_{\beta\in \Gamma^+(\alpha)}G_{\alpha\beta}=
\bigcap_{\beta\in \Gamma^+(\alpha)}{\cent H {n_\beta}}=\cent H{\langle n_\beta\mid \beta\in \Gamma^+(\alpha)\rangle}.$$
Since $\Gamma$ is connected, we deduce $N=\langle n_\beta\mid \beta\in \Gamma^+(\alpha)\rangle$. Hence $G_\alpha^{+[1]}= \cent H N=1$.  In other words, $G_\alpha$ acts faithfully on $\Gamma^+(\alpha)$.

When $G$ has O'Nan-Scott type $\mathrm{HA}$, $\mathrm{TW}$, $\mathrm{HS}$ and $\mathrm{HC}$, the proof follows by the fact that $G$ contains a normal regular subgroup and by the fact that the non-trivial orbital graphs of $G$ are connected.
\end{proof}

\section{Primitive groups of SD type}\label{sec:diagonal}

We start by recalling the structure of the finite primitive groups 
of SD type. This will also allow us to set up the notation for this
section.

Let $\ell\geq 1$ and let $T$ be a non-abelian simple group. Consider
the group $N=T^{\ell+1}$ and $D=\{(t,\ldots,t)\in N\mid t\in T\}$, a
diagonal subgroup of $N$. Set $\Omega:=N/D$, the set of 
right cosets of $D$ in $N$. Then $|\Omega|=|T|^\ell$. Moreover we
may identify each element $\omega\in \Omega$ with an element
of $T^\ell$ as follows: the right coset
$\omega=D(\alpha_0,\alpha_1,\ldots,\alpha_\ell)$
contains a unique element whose first coordinate is $1$, namely, the
element
$(1,\alpha_0^{-1}\alpha_1,\ldots,\alpha_0^{-1}\alpha_\ell)$. We
choose this
distinguished coset representative.
Now the element $\varphi$ of $\Aut(T)$ acts on $\Omega$ by
$$
D(1,\alpha_1,\ldots,\alpha_\ell)^\varphi=D(1,\alpha_1^\varphi,\ldots,\alpha_\ell^\varphi).
$$
Note that this action is well-defined because $D$ is
$\Aut(T)$-invariant.
Next, the element $(t_0,\ldots,t_\ell)$ of $N$ acts on $\Omega$ by
$$
D(1,\alpha_1,\ldots,\alpha_\ell)^{(t_0,\ldots,t_\ell)}=
D(t_0,\alpha_1t_1,\ldots,\alpha_\ell
  t_\ell)=D(1,t_0^{-1}\alpha_1t_1,\ldots,t_0^{-1}\alpha_\ell t_\ell).$$
Observe that the action induced by  $(t,\ldots,t)\in N$ on $\Omega$ is
the same as the action induced by the inner automorphism 
corresponding to the conjugation by $t$.
Finally, the element $\sigma$ in $\Sym(\{0,\ldots,\ell\})$  acts
on $\Omega$ simply by permuting the coordinates. Note that this action
is well-defined because $D$ is $\Sym(\ell+1)$-invariant.

The set of all permutations we described generates a group $W$ isomorphic
to $$T^{\ell+1}\cdot (
\Out(T)\times \Sym(\ell+1)).$$ A subgroup $G$ of $W$ containing the socle 
$N$ of $W$ is primitive if either $\ell=2$ or $G$  acts primitively by 
conjugation on the $\ell+1$ simple direct factors of $N$, see~\cite[Theorem~4.5A]{DM}. The group $G$ is said to be primitive of SD type, when the second case occurs, that is, $N\unlhd G\le W$ and $G$ acts primitively by conjugation on the $\ell+1$ simple direct factors of $N$.

 Write 
$$
M=\{(t_0,t_1,\ldots,t_\ell)\in N\mid t_0=1\}.
$$ 
Clearly, $M$ is a normal subgroup of $N$ acting
regularly on $\Omega$. Since the stabilizer in $W$ of the point
$D(1,\ldots,1)$ is $\Sym(\ell+1)\times \Aut(T)$, we obtain
$$
W=(\Sym(\ell+1)\times \Aut(T))M.
$$ 
Moreover,  every element $x\in W$
can be written uniquely as $x=\sigma\varphi m$, with $\sigma\in \Sym(\ell+1)$,
$\varphi\in\Aut(T)$ and $m\in M$.

\begin{thrm}\label{SD}
Let $G$ be a primitive group on $\Omega$ of SD type, let $\alpha$ and $\beta$ be two distinct elements from $\Omega$ and consider the action of $G$ on the orbital graph determined by $(\alpha,\beta)$. Then one of the following holds
\begin{itemize}
\item $G_\alpha^{+[1]}=1$,
\item $|G_{\alpha}^{+[1]}|=\ell+1$, $G_{\alpha}^{+[1]}\le \Sym(\ell+1)$ and $G_{\alpha}^{+[1]}$ acts regularly on $\{1,\ldots,\ell+1\}$. Moreover, let $(t_0,t_1,\ldots,t_\ell)\in N$ with $t_0=1$ and $\beta=D(t_0,t_1,\ldots,t_\ell)$. The mapping $G_\alpha^{+[1]}\to T$ defined by $\sigma\to t_{0^{\sigma^{-1}}}$ is a group homomorphism. 
\end{itemize}
\end{thrm}
\begin{proof}
We use the notation that we have established above. Without loss of generality we may assume that $$\alpha=D(1,1,\ldots,1).$$ Write $$\beta:=D(1,t_1,\ldots,t_\ell),$$ for some $t_1,\ldots,t_\ell\in T$. We set $t_0:=1$, in particular, $\beta=D(t_0,t_1,\ldots,t_\ell)$. This notation will make the last part of our proof easier to follow.

 Let $\varphi\in G_\alpha^{+[1]}\cap \Aut(T)$. For each $t\in T$, we let $\iota_t\in \Aut(T)\le W_\alpha$ denote the permutation on $\Omega$ induced by the conjugation via $t$. Observe that $\iota_t\in G\cap W_\alpha=G_\alpha$, for every $t\in T$, because $T^{\ell+1}=N\le G$. As $\varphi\in G_\alpha^{+[1]}$ and $\iota_t\in G_\alpha$, we deduce that $\varphi$  fixes $\beta^{\iota_t}$, for every $t\in T$. This means that
\begin{equation}\label{eq1}
D(1,t_1^t,\ldots,t_\ell^t)=\beta^{\iota_t}=\beta^{\iota_t})^\varphi=D(1,t_1^t,\ldots,t_\ell^t)^\varphi=D(1,t_1^{t\varphi},\ldots,t_\ell^{t\varphi}),
\end{equation}
for every $t\in T$. Since $\beta\ne \alpha$, there exists $i\in \{1,\ldots,\ell\}$ with $t_i\ne 1$. Now,~\eqref{eq1} gives $t_i^{t\varphi}=t_i^{t}$, for every $t\in T$. Therefore $$\varphi\in \bigcap_{t\in T}\cent {\Aut(T)}{t_i^t}=\cent {\Aut(T)}{\langle t_i^t\mid t\in T\rangle}=\cent {\Aut(T)}{T}=1.$$
This shows that 
\begin{equation}\label{eq:2}G_\alpha^{+[1]}\cap \Aut(T)=1.
\end{equation}

Now, $G_\alpha^{+[1]}$ is a normal subgroup of $G_\alpha$.  Since $G_\alpha$ acts primitively as a group of permutations on the $\ell+1$ simple direct factors of $T^{\ell+1}$, we obtain that either $G_\alpha^{+[1]}$ projects trivially on $\Sym(\ell+1)$ or $G_\alpha^{+[1]}$ projects to a transitive subgroup of $\Sym(\ell+1)$.
If $G_\alpha^{+[1]}$ projects trivially on $\Sym(\ell+1)$, then $G_\alpha^{+[1]}\le \Aut(T)$ and hence $G_\alpha^{+[1]}=1$ by~\eqref{eq:2}. (In particular, in this case Conjecture~\ref{q:1} part~{\bf (i)} holds true). Therefore, for the rest of this proof we assume that 
\begin{center}$G_\alpha^{+[1]}$ projects to a transitive subgroup of $\Sym(\ell+1)$.\end{center}

Observe now that $\iota(T)=\{\iota_t\mid t\in T\}\le G_\alpha$ (because $T^{\ell+1}=N\le G$) and that $\iota(T)\lhd G_\alpha$ (because $W_\alpha=\Aut(T)\times \Sym(\ell+1)$ and $G_\alpha\le W_\alpha$). As $G_\alpha^{+[1]}\cap \Aut(T)=1$, we deduce $G_\alpha^{+[1]}\cap \iota(T)=1$. Since $G_{\alpha}^{+[1]}$ and $\iota(T)$ are both normal in $G_\alpha$, we deduce that  $G_\alpha^{+[1]}$ centralizes $\iota(T)$. Since $\cent {W_\alpha}{\iota(T)}=\Sym(\ell+1)$, we get $G_\alpha^{+[1]}\leq \Sym(\ell+1)$. Therefore $G_\alpha^{+[1]}$ is a transitive subgroup of $\Sym(\ell+1)$.

Next we show that $G_\alpha^{+[1]}$ is a regular subgroup of $\Sym(\ell+1)$. Let $H:=\norm G{T_0}$, where $T_0$ is the first simple direct factor of the socle $N$. Now, $H$ acts transitively on $\Omega$, because $N\le H$, and $H$ contains the normal regular subgroup $M=T_1\times \cdots \times T_\ell$. Therefore, by Proposition~\ref{prop:1} applied to $H$, we deduce $H_\alpha^{+[1]}=1$. As $|G:H|=|G:\norm G{T_0}|=\ell+1$ and $H_\alpha^{+[1]}=H\cap G_\alpha^{+[1]}$, we get $|G_\alpha^{+[1]}|\le \ell+1$. Since $G_\alpha^{+[1]}$ is a transitive subgroup of $\mathrm{Sym}(\ell+1)$, $G_\alpha^{+[1]}$ is a regular subgroup of $\Sym(\ell+1)$ and 
$$\ell+1=|G_\alpha^{+[1]}|.$$

\smallskip

We need to recall in detail the action of $\Sym(\ell+1)$ on $\Omega$. Given $\sigma\in \Sym(\ell+1)$ and $\omega=D(x_0,x_1,\ldots,x_\ell)\in \Omega$, we have
\begin{equation}\label{eq:3}\omega^\sigma=D(x_{0^{\sigma^{-1}}},x_{1^{\sigma^{-1}}},\ldots,x_{\ell^{\sigma^{-1}}}).
\end{equation}
The element $\sigma$ in the right hand side of~\eqref{eq:3} appears as $\sigma^{-1}$ to guarantee that this is a right action.

Recall $\beta=D(1,t_1,\ldots,t_\ell)=D(t_0,t_1,\ldots,t_\ell)$. We now define a mapping
\begin{center}
    \begin{tikzcd}[cells = {nodes={minimum width=3.5em, inner xsep=0pt}},
                   row sep=0pt]
w:G_\alpha^{+[1]} \ar[r]           &   T,   \\
\sigma      \ar[r, mapsto]   & t_{0^{\sigma^{-1}}}.
\end{tikzcd}
\end{center} 
In other words, in the light of~\eqref{eq:3}, $w(\sigma)$ is the first coordinate of $(t_0,t_1,\ldots,t_\ell)^\sigma=(t_{0^{\sigma^{-1}}},t_{1^{\sigma^{-1}}},\ldots,t_{\ell^{\sigma^{-1}}})$. 

Let $\sigma,\tau\in G_\alpha^{+[1]}$. 
Since $\tau$ fixes $\beta$, we have $$\beta=\beta^\tau=D(t_{0^{\tau^{-1}}},t_{1^{\tau^{-1}}},\ldots,t_{\ell^{\tau^{-1}}})$$
and since $\sigma \tau$ fixes $\beta$, we have also
$$\beta=\beta^{\sigma\tau}=D(t_{0^{(\sigma\tau)^{-1}}},t_{1^{(\sigma\tau)^{-1}}},\ldots,t_{\ell^{(\sigma\tau)^{-1}}}).$$
In other words, the two $(\ell+1)$-tuples $(t_{0^{\tau^{-1}}},t_{1^{\tau^{-1}}},\ldots,t_{\ell^{\tau^{-1}}})$ and 
$(t_{0^{(\sigma\tau)^{-1}}},t_{1^{(\sigma\tau)^{-1}}},\ldots,t_{\ell^{(\sigma\tau)^{-1}}})$ differ only by the left multiplication by an element of $D$. Therefore, there exists $t\in T$ such that
\begin{equation}\label{eq:4}
(tt_{0^{\tau^{-1}}},tt_{1^{\tau^{-1}}},\ldots,tt_{\ell^{\tau^{-1}}})=(t_{0^{(\sigma\tau)^{-1}}},t_{1^{(\sigma\tau)^{-1}}},\ldots,t_{\ell^{(\sigma\tau)^{-1}}}).
\end{equation}
By checking the first coordinates in~\eqref{eq:4}, we obtain 
$$tt_{0^{\tau^{-1}}}=t_{0^{(\sigma\tau)^{-1}}}.$$ Moreover, by comparing the coordinate appearing in position $0^{\tau}$ on the left-hand side and on the right-hand side of~\eqref{eq:4}, we deduce $tt_{(0^{\tau})^{\tau^{-1}}}=t_{(0^\tau)^{(\sigma\tau)^{-1}}}$, that is, $$t=tt_0=t_{0^{\sigma^{-1}}}.$$ Putting these two equations together, we obtain
$$w(\sigma)w(\tau)=t_{0^{\sigma^{-1}}}t_{0^{\tau^{-1}}}=tt_{0^{\tau^{-1}}}=t_{0^{(\sigma\tau)^{-1}}}=w(\sigma\tau).$$
This proves that our mapping $w:G_\alpha^{+[1]}\to T$ is a group homomorphism. 
\end{proof}

The proof of Proposition~\ref{SD} hints to the fact that in Conjecture~\ref{q:1} we do need the alternative~{\bf (ii)}. We show that this is indeed the case in the next example.

\begin{example}\label{examplemrone}{\rm
 Let $p$ be a prime number, let $k\ge 7$ be a positive integer and let $r$ be a primitive prime divisor of $p^{k}-1$, that is, $r$ is relatively prime to $p^i-1$ for each $i\in \{1,\ldots,k-1\}$. As $k\ge 7$, the existence of $r$ is guaranteed by Zsigmondy's theorem.
 
  Let $H:=V\rtimes C$ be the affine primitive group of degree $p^{k}$, where $V$ is an elementary abelian $p$-group of order $p^{k}$ and where $R$ is a cyclic group of order $r$. (We use an additive notation for $V$.) Let $T$ be a non-abelian simple group containing a cyclic subgroup $P$ of order $p$ and with $\cent TP=P$ and let $$w:V\to P$$ be an arbitrary surjective homomorphism.

We denote by $T^V$ the set of all functions from $V$ to $T$. Observe that $T^V$ is a group isomorphic to the Cartesian product of $|V|=p^k$ copies of $T$. We denote the elements of $T^V$ as functions $f:V\to T$.

We let $G$ be primitive group of diagonal type $$T^V\rtimes H.$$
Recall that the elements of $\Omega$ are right cosets of $D$ in $T^V$, where $D$ is the diagonal subgroup of $T^V$, that is, $D=\{f\in T^V\mid f\textrm{ is constant}\}$. Let $b=(t_v)_{v\in V}\in T^V$ where $t_v=\omega(-v)$, for every $v\in V$.

Let $\alpha:=D$, let $\beta:=Db$ and consider the orbital graph determined by $(\alpha,\beta)$.  Let $v\in V$. Then
$$b^v(x)=b(x-v)=w(-x+v)=w(-x)w(v)=w(v)w(-x)=w(v)b(x),\quad\forall x\in V.$$
Thus $b^v=w(v)b$ and hence $\beta^v=Db^v=Db=\beta$. This shows $V\le G_{\alpha\beta}$. Since $V\unlhd G_\alpha$, we deduce $G_\alpha^{+[1]}\le V$. Now, from Proposition~\ref{SD}, we have $G_\alpha^{+[1]}=V$. 

Thus $G_\alpha=(T\times R)G_{\alpha}^{+[1]}$ and 
$$G_{\alpha\beta}=(G_{\alpha\beta}\cap (T\times R))G_{\alpha}^{+[1]}.$$
Let $\varphi:=th\in G_{\alpha\beta}\cap (T\times R)$, with $t\in T$ and $h\in R$.  Observe that
\begin{equation}\label{eq:lastone}
b^{th}(0)=b(0^{h^{-1}})^t=b(0)^{t}=w(0)^t=1^t=1=w(0)=b(0).\end{equation}
Since $Db=\beta=\beta^{th}=Db^{th}=Db$, from~\eqref{eq:lastone} we get $b^{th}=b$.

For every $v\in \mathrm{Ker}(w)\le V$, we have
$$w(-v^{h^{-1}})^t=(b(v^{h-1}))^t=b^{th}(v)=b(v)=w(-v)=1.$$
Therefore, $w(v^{h^{-1}})=1$. This gives $\mathrm{Ker}(w)^{h^{-1}}=\mathrm{Ker}(w)$. As $\dim\mathrm{Ker}(w)=k-1\ne 0$ and as $R$ is a cyclic group of prime order acting irreducibly on the vector space $V$,  we deduce $h=1$. This shows that $\varphi=t\in T$.

Now, $b^t=b$ if and only if $t$ centralizes all the coordinates of $b$. In other words, $t\in \cent T P=P$. Summing up,
$$G_{\alpha\beta}=P\times V\hbox{ and }G_\alpha^{+[1]}=V.$$ 
Thus $|G_{\alpha\beta}:G_\alpha^{+[1]}|=|P|=p$ and $|G_\alpha^{+[1]}|=p^{k}$. However, we cannot bound the cardinality of $G_\alpha^{+[1]}$ with $p$ only.}
\end{example}

\section{The example for the Cameron and Fomin question}\label{sec:example}
Our construction is quite elaborate and requires a number of ingredients:
\begin{itemize}
\item let $A$ be a non-abelian simple group,
\item let $T$ be a non-abelian simple group containing a subgroup $Q$ with $Q= A\times A$ and with
$\cent T Q=1$,
\item let $H$ be a group containing $A$ with $A$ maximal in $H$, $A$ core-free in $H$ and, in the faithful permutation action of $H$ on the right cosets of $A$, there exist two points whose setwise stabilizer is the identity.
\end{itemize}
We observe that there are groups $A,T,Q$ and $H$ satisfying the hypothesis above. For instance, we may take $A:=\mathrm{Alt}(5)$, $T:=\mathrm{Alt}(10)$, $Q=\mathrm{Alt}(5)\times\mathrm{Alt}(5)\le T$ and $H:=\mathrm{PSL}_2(p)$ where $p$ is a  prime number with $p\ge 61$ and  $p\equiv \pm1\pmod{10}$. Clearly, $A$ and $T$ are non-abelian simple and $\cent T Q=1$. Moreover, using the hypothesis on $p$, $A$ is a maximal subgroup in the Aschbacher class $\mathcal{S}$ of $T$, see for instance~\cite[Table~$8.2$]{BHR}. Using the fact that $p>19$, we see from~\cite[Table~$1$]{Burness} that the base size of $T$ in the action on the right cosets of $A$ is $2$. Actually, from the arguments in~\cite{Burness}, it follows that whenever $p\ge 61$, there exist two points whose setwise stabilizer is the identity.

Let $V:=A^{|H:A|}$ and let $L:=V\rtimes H$ be a primitive group of TW type with regular socle $V$ and with point stabilizer $H$. The fact that $L$ is primitive in its action on $V$ follows from~\cite[Lemma~$4.7$A]{DM}. 

As in Example~\ref{examplemrone}, we denote by $T^V$ the set of all functions from $V$ to $T$. 
We let $G$ be primitive group of diagonal type $$T^V\rtimes L.$$
The elements of $\Omega$ are right cosets of $D$ in $T^V$, where $D$ is the diagonal subgroup of $T^V$, that is, $D=\{f\in T^V\mid f\textrm{ is constant}\}$.

Relabelling the elements in the domain $\{1,\ldots,|H:A|\}$, we may suppose that $1,2$ is a base for the action of $H$ and the setwise stabilizer of $\{1,2\}$ in $H$ is the identity. We define the group homomorphism
\begin{center}
    \begin{tikzcd}[cells = {nodes={minimum width=3.5em, inner xsep=0pt}},
                   row sep=0pt]
w:V=A^{|H:A|} \ar[r]           &   Q\le T,   \\
(a_1,a_2,\ldots,a_{|H:A|})      \ar[r, mapsto]   & (a_1,a_2).
\end{tikzcd}
\end{center} 

 Let $b\in T^V$ with $b(v)=\omega(v^{-1})$, for every $v\in V$. Let $\alpha:=D$, let $\beta:=Db$ and consider the orbital graph determined by $(\alpha,\beta)$.  Let $v\in V$. Then
$$b^v(x)=b(xv^{-1})=w((xv^{-1})^{-1})=w(vx^{-1})=w(v)w(x^{-1})=w(v)b(x),\quad\forall x\in V.$$
Thus $b^v=w(v)b$ and hence $\beta^v=Db^v=Db=\beta$. This shows $V\le G_{\alpha\beta}$. Since $V\unlhd G_\alpha$, we deduce $G_\alpha^{+[1]}\le V$. Now, from Proposition~\ref{SD}, we have $G_\alpha^{+[1]}=V$. 

Thus $G_\alpha=T\times L=T\times HV=T\times HG_{\alpha}^{+[1]}=(T\times H)G_{\alpha}^{+[1]}$ and 
\begin{equation}\label{display}
G_{\alpha\beta}=(G_{\alpha\beta}\cap (T\times H))G_{\alpha}^{+[1]}.
\end{equation}
Let $\varphi:=th\in G_{\alpha\beta}\cap (T\times H)$, with $t\in T$ and $h\in H$.  Observe that
\begin{equation}\label{eq:lastoneone}
b^{th}(1)=b(1^{h^{-1}})^t=b(1)^{t}=w(1)^t=(1,1)^t=(1,1)=w(1)=b(1).\end{equation}
Since $Db=\beta=\beta^{th}=Db^{th}=Db$, from~\eqref{eq:lastoneone} we get $b^{th}=b$.

For every $v=(1,1,a_3,\ldots,a_{|H:A|})\in \mathrm{Ker}(w)\le A^{|H:A|}=V$, we have
$$w((v^{-1})^{h^{-1}})^t=(b(v^{h^{-1}}))^t=b^{th}(v)=b(v)=w(v^{-1})=(1,1).$$
Therefore, $w((v^{-1})^{h^{-1}})=(1,1)$, for every $v=(1,1,a_3,\ldots,a_{|H:A|})\in\mathrm{Ker}(w)\le A^{|H:A|}=V$. This gives that $h$ fixes setwise $\{1,2\}$ and hence $h=1$, by our assumption on the permutation action of $H$ on the right cosets of $A$.
This shows that $\varphi=t\in T$.

Now, $b^t=b$ if and only if, for every $v=(a_1,a_2,\ldots,a_{|H:A|})\in V$, we have $b^t(v)=b(v)$, that is,
$(b(v))^t=b(v)$. This yields
$$(a_1^{-1},a_2^{-1})^t=(w(v^{-1}))^t=(b(v))^t=b(v)=w(v^{-1})=(a_1^{-1},a_2^{-1}).$$
Since this holds for each $(a_1,a_2)\in A\times A=Q\le T$, we deduce $t\in \cent TQ=1$. This shows that $\varphi=1$.

As $\varphi$ was an arbitrary element in $G_{\alpha\beta}\cap (T\times H)$, we get $G_{\alpha\beta}\cap (T\times H)=1$ and hence $G_{\alpha\beta}=G_\alpha^{+[1]}$ from~\eqref{display}. Therefore $G_{\alpha\beta}\unlhd G_\alpha$.

\end{document}